\documentclass[11pt,a4paper]{article}
\usepackage[utf8]{inputenc}

\usepackage{amsmath}
\usepackage{amsfonts}
\usepackage{amssymb}
\usepackage{amsthm}

\usepackage{bbm}
\usepackage{dsfont}
\usepackage{color}

\usepackage{hyperref}
\usepackage{cleveref}
\usepackage{enumitem}
\usepackage{microtype}

\usepackage{stmaryrd}
\makeatletter
\newcommand{\norm}[1]{\left\lVert#1\right\rVert}
\newcommand\restr[2]{{
  \left.\kern-\nulldelimiterspace 
  #1 
  \right|_{#2} 
  }}
\DeclareMathOperator\supp{supp}

\DeclareMathOperator\Tr{Tr}

\newtheorem{thm}{Theorem}[section]
\newtheorem{hyp}[thm]{Hypothesis}

\newtheorem{rem}[thm]{Remark}
\newtheorem{lem}[thm]{Lemma}
\newtheorem{prop}[thm]{Proposition}

\@addtoreset{equation}{section}
\makeatother

\usepackage{authblk}

\begin{document}

\title{Mesoscopic central limit theorem for general $\beta$-ensembles}
\date{}

\author{Florent Bekerman\thanks{\texttt{bekerman@mit.edu}}}
\affil{MIT, Department of Mathematics.}
\author{Asad Lodhia \thanks{\texttt{lodhia@math.mit.edu}}}
\affil{MIT, Department of Mathematics.} 

\maketitle

\begin{abstract}
\noindent We prove that the linear statistics of eigenvalues of $\beta$-log gasses satisfying the one-cut and off-critical assumption with a potential $V \in C^6(\mathbb{R})$ satisfy a central limit theorem at all mesoscopic scales $\alpha \in (0; 1)$. We prove this for compactly supported test functions $f \in C^5(\mathbb{R})$  using loop equations at all orders along with rigidity estimates.
\end{abstract}

\section{Introduction}
\noindent We consider a system of $N$ particles on the real line distributed according to a density proportional to 
\begin{equation*}
\prod_{1\leq i < j \leq N} \lvert \lambda_i - \lambda_j \rvert^\beta e^{- N \sum V(\lambda_i)} \prod d\lambda_i \ ,
\end{equation*}
\noindent where $V$ is a continuous potential and $\beta > 0$. This system is called the $\beta$-log gas, or general $\beta$-ensemble and for classical values of $ \beta \in \{ 1,2,4 \}$, this distribution corresponds to the joint law of the eigenvalues of symmetric, hermitian or quaternionic random matrices with density proportional to $e^{-N \Tr V(M)} dM$ where $N$ is the size of the random matrix $M$. 

\noindent Recently, great progress has been made to understand the behaviour of $\beta$-log gasses. At the microscopic scale, the eigenvalues exhibit a universal behaviour (see \cite{BFG}, \cite{BEYI}, \cite{BEYII}, \cite{B}) and the local statics of the eigenvalues are described by the $Sine_\beta$ process in the bulk and the Stochastic Airy Operator at the edge (see \cite{VV} and \cite{RRV} for definitions). At the macroscopic level, the eigenvalues satisfy a central limit theorem and the re-centered linear statistics of the eigenvalues converge towards a Gaussian random variable. This was first proved in \cite{J} for polynomial potentials satisfying the one-cut assumption. In \cite{BGI}, the authors derived a full expansion of the free energy in the one-cut regime from which they deduce the central limit theorem for analytic potentials. The multi-cut regime is more complicated and in this setting, the central limit theorem does not hold anymore for all test functions (see \cite{BGII}, \cite{ShI}). In this article, we consider the scale between microscopic and macroscopic called the mesoscopic regime. Specifically, we study the linear fluctuations of the eigenvalues of general $\beta$-ensembles at the mesoscopic scale; we prove that for $\alpha \in (0 ; 1)$ fixed, $f$ a smooth function (whose regularity and decay at infinity will be specified later), and $E$ a fixed energy level
\begin{equation*}
\sum_{i=1}^N f\big(N^\alpha(\lambda_i - E)\big) \ - N \int f(N^\alpha(x - E)) d\mu_V(x)
\end{equation*}
converges towards a Gaussian random variable. 

\noindent Interest in mesoscopic linear statistics has surged in recent years. Results in this field of study were obtained in a variety of settings,  for Gaussian random matrices \cite{BdMK99a, FKS13},  and for invariant ensembles \cite{BD14,L16}. In many cases the results were shown at all scales $\alpha \in (0;1)$, often with the use of distribution specific properties. In more general settings, the absence of such properties necessitates other approaches to obtain the limiting behaviour at the mesoscopic regime. For example, an early paper studying mesoscopic statistics for Wigner Matrices was \cite{BdMK99b}, here the regime studied was $\alpha \in (0;\frac{1}{8})$, later using improved local law results this was pushed to $ \alpha \in (0 ; \frac{1}{3})$ \cite{LS15}, and recent work has pushed this to all scales \cite{HK16}. 

\noindent Extending these results to general $\beta$-ensembles  is a natural step. We also prove convergence at all mesoscopic scales. The proof of the main Theorem relies on the analysis of the loop equations from which we can deduce a recurrence relationship between moments, and the rigidity results from \cite{BEYI}, \cite{BEYII} to control the linear statistics. Similar results have been obtained before in \cite[Theorem 5.4]{bourgade}. There, the authors showed the mesoscopic CLT in the case of a quadratic potential, for small $\alpha$ (see Remark 5.5). 

\noindent  In Section 1, we introduce the model and  recall  some background results and Section 2 will be dedicated to the proof of \Cref{theoreme}.

\subsection{Definitions and Background}

\noindent  We consider the general $\beta$-matrix model.
\noindent For a   potential $V : \mathbb{R} \longrightarrow \mathbb{R}$  and $\beta > 0$, we denote the measure on $\mathbb{R}^N$
\begin{equation}
\mathbb{P}_{V}^N \ (d\lambda_1,\cdots,d\lambda_N):= \frac{1}{Z_{V}^N}\prod_{1\leq i < j \leq N} \lvert \lambda_i - \lambda_j \rvert^\beta e^{- N \sum V(\lambda_i)}  \prod d\lambda_i \ ,
\end{equation}

\noindent with  \[{Z_{V}^N}  = \int \prod_{1\leq i < j \leq N} \lvert \lambda_i - \lambda_j \rvert^\beta e^{- N\sum V(\lambda_i)}  \prod d\lambda_i  . \]  

\noindent It is well known that under $\mathbb{P}_{V}^N$ the empirical measure of the eigenvalues converge towards an equilibrium measure:

\begin{thm}
Assume that $V : \mathbb{R}  \longrightarrow \mathbb{R}$ is continuous and that 

\begin{equation*}
\liminf_{x \rightarrow \infty} \frac{V(x)}{\beta \log |x|} >1 .
\end{equation*} 

\noindent Then the energy defined by 
\begin{equation}\label{energy}
E(\mu) = \iint\left(\frac{V(x_1) + V(x_2)}{2}  - \frac{\beta}{2} \log |x_1-x_2 | \right) d\mu(x_1) d\mu(x_2)
\end{equation}

\noindent has a unique global minimum on the space $\mathcal{M}_1(\mathbb{R} )$ of probability measures on $\mathbb{R}$.

\noindent Moreover, under $\mathbb{P}_{V}^N$ the normalized empirical measure $L_N = N^{-1} \sum_{i=1}^N \delta_{\lambda_i}$ converges almost surely and in expectation towards the unique probability measure $\mu_{V} $ which minimizes the energy.

\noindent Furthermore, $\mu_V$ has compact support $A$ and is uniquely determined by the existence of a constant $C$ such that:
\begin{equation*}
\beta \int \log |x-y|  d\mu_{V}(y) - V(x) \leq C  \ ,
\end{equation*}
with equality almost everywhere on the support. The support of  $\mu_{V}$  is a union of intervals  $A = \underset {{0\leq h \leq g}}{\bigcup} [\alpha_{h,-} ; \alpha_{h,+}]$  with $\alpha_{h,-} < \alpha_{h,+}$  and if $V$ is smooth on a neighbourhood of $A$,
\begin{equation*}
\frac{d\mu_{V}}{dx}= S(x)\prod_{h=0}^g \sqrt{\lvert x - \alpha_{h,-} \rvert \lvert x - \alpha_{h,+} \rvert} \ , 
\end{equation*}
\noindent with $S$  smooth on a neighbourhood of $A$.
\end{thm}

\subsection{Results}
\begin{hyp}\label{hypo} For what proceeds, we assume the following
\begin{itemize}
\item  $V$  is  continuous and goes to infinity faster than  $\beta \ \log \rvert x \lvert$.
\item The support of  $\mu_{V}$ is a connected interval   $A= [a ; b]$  and
\begin{equation*}
\frac{d\mu_{V}}{dx}= \rho_V(x)= S(x) \sqrt{(b-x)(x-a)}  \ \ \ \ with \  S > 0 \ on \ [a ; b] .
\end{equation*}
\item The function $ V(\cdot) - \beta \int \log |\cdot-y|  d\mu_{V}(y) $ achieves its minimum on the support only.

\end{itemize}
\end{hyp}

\begin{rem}
The second and third assumptions are typically known as the one-cut and off-criticality assumptions. In the case where the  support of the equilibrium measure is no longer  connected, the macroscopic central limit theorem does not hold anymore in generality (see \cite{BGII} , \cite{ShI}). Whether the  theorem holds for critical potentials is still an open question.
\end{rem}

\begin{thm}\label{theoreme}
Let $ 0 < \alpha <  1$ ,  $E$ a point in the bulk $(a ; b)$, $V\in C^6(\mathbb{R})$  and  $f \in C^5(\mathbb{R})$ with compact support. Then, under $\mathbb{P}_{V}^N $ 
\begin{equation*}
\sum_{i=1}^N f\big(N^\alpha(\lambda_i - E)\big) - N \int f(N^\alpha(x - E)) d\mu_V(x) \xrightarrow{\ \ \ \mathcal{M} \ \ \ } \mathcal{N}(0 , \sigma^2_f) \  ,
\end{equation*}
where the  convergence holds in  moments (and thus, in distribution), and
\begin{equation*}
\sigma^2_f = \frac{1}{2 \beta \pi^2 }\iint \left(\frac{f(x) - f(y)}{x-y}\right)^2 dx dy \ .
\end{equation*}
\end{thm}

\noindent Note that, as in the macroscopic central limit theorem, the variance is universal in the potential with a multiplicative factor proportional to $\beta$. Interestingly and in contrast with the macroscopic scale, the limit is always centered.

\noindent The proof relies on an explicit computation of the moments of the linear statistics. We will use two tools: optimal rigidity for the  eigenvalues of beta-ensembles to provide a bound on the linear statistics  (as in \cite{BEYI}, \cite{BEYII}) and the  loop equations at all orders to derive a recurrence relationship between the moments. 

\section*{Acknowledgements} The authors would like to thank Alice Guionnet for the fruitful discussions and the very helpful comments.

\section{Proof of \Cref{theoreme}}
For what follows, set
\begin{equation*}
L_N = \frac{1}{N}\sum_{i} \delta_{\lambda_i}, \quad M_N = \sum_{i=1}^N \delta_{\lambda_i}  - N \mu_V.
\end{equation*}
and for a measure $\nu$ and an integrable function $h$  set 
\begin{equation*}
\nu(h) = \int h d\nu \qquad  \mathrm{and}  \qquad \tilde{\nu}(h) = \int h d\nu -  \mathbb{E}_V^N \Big( \int h d\nu \Big),
\end{equation*}
when $\nu$ is random and where $\mathbb{E}_V^N$ is expectation with respect to $\mathbb{P}_V^N$. Further $f$ will by any function as in \Cref{theoreme}, and 
\begin{equation*}
\tilde{f}(x) := f(N^\alpha(x-E)).
\end{equation*}
Finally, for any function $g \in C^p(\mathbb{R})$, let
\begin{equation*}
\|g \|_{C^p(\mathbb{R})} := \sum_{l=0}^p\sup_{x\in \mathbb{R}} |g^{(l)}(x)|,
\end{equation*}
when it exists.

\subsection{Loop Equations}
\noindent To prove the convergence, we use the loop equations at all orders. Loop equations have been used previously to derive recurrence relationships between correlators and derive a full expansion of the free energy for $\beta$-ensembles in \cite{ShI}, \cite{BGII}, and  \cite{BGI} (from which the authors also derive a macroscopic central limit theorem). The  first loop equation was used to prove the central limit theorem at the macroscopic scale in \cite{J} and used subsequently in \cite{bourgade}. Here, rather than using the first loop equation to control the Stieltjes transform as  in \cite{J} and \cite{bourgade}, we rely on the analysis of the loop equations at all orders  to compute directly the moments.

\begin{prop}\label{loopequations}
\noindent Let  $h$, $h_1, h_2, \cdots$ be a sequence of functions in $C^1(\mathbb{R})$. Define 
\begin{equation}
 F_1^N(h)  = \frac{\beta}{2}\iint \frac{h(x)-h(y)}{x-y} dL_N(x) dL_N(y) - L_N (h V')  + \frac{1}{N} \big(1 - \frac{\beta}{2}\big) L_N(h')   
\end{equation}
and for all $k \geq 1$
\begin{equation}
 F^N_{k+1}(h, h_1, \cdots , h_k)  = F^N_k(h, h_1, \cdots, h_{k-1}) \tilde{M}_N(h_k) + \left(\prod_{l=1}^ {k-1} \tilde{M}_N(h_l)\right)L_N(h h'_k)
\end{equation}
where the product is equal to $1$ when $k=1$. Then we have for all $k \geq 1$
\begin{equation}
\mathbb{E}_V^N \big( F^N_k(h, h_1, \cdots , h_{k-1}) \big) =  0.
\end{equation}
\end{prop}

\begin{proof}
\noindent  The first loop equation is derived by integration by parts. We derive the loop equation at order $k+1$ from the one at order $k$ by replacing $V$ by  $V + \delta h_k$ and differentiating at $\delta = 0$.
\end{proof}

\noindent It will be easier to compute recursively the moments by re-centering the first loop equation. To that end, define the operator $\Xi$ acting  on  smooth functions $h:\mathbb{R} \longrightarrow \mathbb{R}$ by
\begin{equation*}
\Xi h(x) =  \beta  \int \  \frac{h(x)-h(y)}{x-y}   d\mu_{V}(y) -  V'(x) h(x) \ .
\end{equation*}

\noindent We then use equilibrium relations to recenter $L_N$ by $\mu_{V}$. Consider for $\delta$ in a neighbourhood of $0$, $\mu_{V,\delta} = (x + \delta h(x))\sharp \mu_{V}$, where for a map $T$ and measure $\mu$, $T \sharp \mu$ refers to the push-forward measure of $\mu$ by $T$. Then by (\ref{energy}) we have $E(\mu_{V,\delta})  \geq E(\mu_{V})$ . By differentiating at $\delta = 0$  we obtain

\begin{equation}\label{equilibre}
\frac{\beta}{2} \iint \frac{h(x)-h(y)}{x-y} d\mu_{V}(x)d\mu_{V}(y) = \int V'(x)f(x) d\mu_{V}(x) \ ,
\end{equation} 
and thus
\begin{multline*}
 \frac{\beta }{2}  \iint \frac{h(x)-h(y)}{x-y}dL_N(x) dL_N(y) -  L_N(h V') = \\
 \frac{1}{N} M_N(\Xi h) +  \frac{\beta}{2 N^2} \iint \frac{h(x)-h(y)}{x-y} dM_N(x) dM_N(y).
\end{multline*}

\noindent Consequently, we can write
\begin{equation}\label{loop1}
F_1^N(h) =  M_N(\Xi h) + \Big(1- \frac{\beta}{2} \Big) L_N(h') +  \frac{1}{N}  \left[\frac{\beta}{2}   \iint  \frac{h(x)-h(y)}{x-y} dM_N(x) dM_N(y)   \right].
\end{equation}

\noindent One of  the key features of the operator $\Xi$ is that it is invertible (modulo constants) in the space of smooth functions. More precisely, we have the following Lemma (see  Lemma 3.2 of \cite{BFG} for the proof):

\begin{lem}{Inversion of  $\Xi$}

\noindent Assume that $V \in C^p(\mathbb{R})$ and satisfies Hypothesis \ref{hypo}. Let $[a ; b]$ denote the support of $\mu_{V}$ and set 
\begin{equation*}
\frac{d\mu_{V}}{dx} = S(x) \sqrt{(b - x)(x - a)}= S(x) \sigma(x) ,
\end{equation*}
\noindent where  $S > 0$ on $[a ; b]$.\\
\noindent Then for any $k\in C^r(\mathbb{R})$ there exists a unique constant $c_k$ and $h\in C^{(r-2)\wedge(p-3)}(\mathbb{R})$ such that
\begin{equation*}
\Xi(h) = k + c_k \ .
\end{equation*}

\noindent Moreover the  inverse is given  by  the following formulas:

\begin{itemize}
\item $ \forall x \in \supp(\mu_{V})$
\begin{equation}\label{inversesup}
h(x) = - \frac{1}{\beta \pi^2 S(x)   } \Big( \   \int_{a}^{b} \frac{k(y) - k(x)}{\sigma(y)(y-x)} dy  \Big) 
\end{equation}
\item $ \forall x \notin \supp(\mu_{V})$
\begin{equation}\label{inverseout}
h(x) = \frac{\beta  \int \frac{h(y)}{x-y} d\mu_{V}(y) - k(x) - c_k}{\beta  \int \frac{1}{x-y} d\mu_{V}(y) - V'(x)} \ .
\end{equation}
\end{itemize}
Note that the definition  (\ref{inverseout}) is proper since $h$ has been defined on the support.

\noindent We shall denote this inverse by  $\Xi^{-1}k$.
\end{lem}
\begin{rem}
For $f$ and $V$ as in \Cref{theoreme},  $p = 6$ and $r = 5$ so $\Xi^{-1}\tilde{f} \in C^3(\mathbb{R})$.
\end{rem}

\noindent In order to bound the linear statistics we use the following lemma to bound $\Xi^{-1}(\tilde{f})$ and its derivatives.

\begin{lem}
\label{bornebounds}
Let $\supp f \subset [-M,M]$ for some constant $M > 0$. For each $p \in \{1,2,3\}$, there is a constant $C > 0$ such that 
\begin{equation}\label{bornesup}
\norm{ \Xi^{-1}(\tilde{f})}_{C^p(\mathbb{R})} \leq C  N^{p \alpha}\log N,
\end{equation}
Moreover, there is a  constant $C$ such that whenever $ N^\alpha|x - E| \geq M + 1$
\begin{equation}\label{borneexsup}
\Big| \Xi^{-1}(\tilde{f})^{(p)} (x) \Big| \leq  \frac{C}{N^\alpha  \big(  (x - E)^{p+1} \wedge 1 \big) },
\end{equation}
\begin{proof}
\noindent We start with \eqref{bornesup}. Using \eqref{inverseout}, we see that $\Xi^{-1}(\tilde{f})$ and its derivatives are clearly uniformly bounded outside $\supp \mu_V$. For $x \in \supp \mu_V$ we  use 
\begin{equation*}
\Xi^{-1}(\tilde{f})(x) = -\frac{N^\alpha}{\beta \pi^2 S(x)} \int_a^b \frac{1}{\sigma(y)} \int_0^1 f'\big(N^\alpha t (x - E) + N^\alpha (1-t) (y - E)\big) dt dy \,
\end{equation*}
so that 
\begin{multline*}
\Xi^{-1}(\tilde{f})^{(p)}(x) = -\frac{1}{\beta \pi^2 } \sum_{l=0}^p \Bigg\{\binom{p}{l} \left(\frac{1}{S}\right)^{(p-l)}(x) \\
\times \int_a^b \frac{N^{(l+1)\alpha}}{\sigma(y)} \int_0^1 t^l f^{(l+1)}\big(N^\alpha t (x - E) + N^\alpha (1-t) (y - E)\big) dt dy \Bigg\}\ .
\end{multline*}
Let $A(x) =  \big\{ (t,y) \in [0;1] \times [a;b] \ , \ N^\alpha  |t (x - E) + (1-t) (y - E)| \leq M  \big\}$. We have
\begin{equation}
\int_0^1 \mathbbm{1}_{A(x)}(t,y) dt \leq \frac{2M}{N^\alpha |x-y|} \wedge  1 \, 
\end{equation}
and thus 
\begin{equation*}
 \int_a^b \frac{N^{(l+1)\alpha}}{\sigma(y)} \int_0^1 | f^{(l+1)}\big(N^\alpha t (x - E) + N^\alpha (1-t) (y - E)\big) | dt dy \ \leq C \log N N^{l \alpha} ,
\end{equation*}
and this proves \eqref{bornesup}. 

\noindent We now proceed with the proof  of \eqref{borneexsup}. First, let $x\in \supp \mu_V$ such that $N^\alpha |x-E| \geq M+1$. The inversion formula \eqref{inversesup} writes 
\begin{equation}\label{invint}
\begin{split}
\Xi^{-1}(\tilde{f})(x) &= -\frac{1}{\beta \pi^2 S(x)} \int_a^b \frac{f(N^\alpha(y-E))}{\sigma(y)(y-x)}  dy \\
&=  -\frac{1}{\beta \pi^2 S(x)} \int_{-M}^{M} \frac{f(u)}{\sigma(E + \frac{u}{N^\alpha})  (u - N^\alpha (x-E))} \ .
\end{split}
\end{equation}
By differentiating this formula, we obtain \eqref{borneexsup} for $x\in \supp \mu_V$. The result for
$x\notin \supp \mu_V$  is obtained similarly using \eqref{inverseout}.  
\end{proof}

\end{lem}

\subsection{Control of the linear statistics}
\noindent We now make use of the strong rigidity estimates proved  in \cite{BEYII} (Theorem 2.4) to control the linear statistics. We recall the result here

\begin{thm}
\label{thm::BEYrigid}
Let $\gamma_i$ the quantile defined by
\begin{equation}
\int_a^{\gamma_i} d\mu_V(x) = \frac{i}{N} .
\end{equation}

\noindent Then, under Hypothesis \ref{hypo} and for all $\xi > 0$ there exists constants $c >0$ such that for $N$ large enough 

\begin{equation*}
\mathbb{P}_V^N \big(|\lambda_i - \gamma_i | \geq N^{-2/3 +\xi } \  \hat{i}^{-1/3} \big) \leq e^{-N^c} \ ,
\end{equation*}

\noindent where $\hat{i} = i \wedge (N+1 - {i})$.
\end{thm}

\noindent We will use the following lemma quite heavily in what proceeds.
\begin{lem}\label{regions}
\noindent 
Let $\gamma_i$ and $\hat{i}$ be as in Theorem \ref{thm::BEYrigid}, let $t\in [0;1]$, and let $\lambda_i$, $i \in \llbracket 1,  N \rrbracket$, be a configuration  of points such that $|\lambda_i - \gamma_i| \leq N^{-2/3 + \xi} \, \hat{i}^{-1/3}$ for $0 < \xi < (1 - \alpha) \wedge \frac{2}{3}$, and let $M > 1$ be a constant.  Define the pairwise disjoint sets:
\begin{align}
J_1 &:= \{ i \in \llbracket 1; N\rrbracket,\, |N^\alpha(\gamma_i - E)| \leq 2M\}, \label{def::J1} \\
J_2 &:= \left \{ i\in J_1^c, | (\gamma_i - E)|  \leq \frac{1}{2} (E-a)\wedge (b-E)\right\}, \label{def::J2}\\
J_3 &:=  J_1^c \cap J_2^c.
\end{align}
The following statements hold:
\begin{enumerate}[label=(\alph*)]
\item For all $i \in  J_1 \cup J_2$ ,  $\hat{i}\geq CN$, for some $C > 0$ that depend only on $\mu_V$ in a neighborhood of $E$, also for all such $i$, $|\gamma_i - \gamma_{i+1}|\leq \frac{C}{N}$ for a constant $C > 0$ depending only on $\mu_V$ in a neighborhood of $E$. \label{regionsHati}
\item Uniformly in all $i \in J_1^c = J_2 \cup J_3 $ and all $t\in [0;1]$,\
\begin{equation}
\label{eqn::outofsupport}
|N^\alpha t (\lambda_i - \gamma_i) + N^\alpha (\gamma_i - E)| > M + 1,
\end{equation}
for large enough $N$. Furthermore, the statement holds true uniformly in $x\in [\gamma_i,\gamma_{i+1}]$ when we substitute $\gamma_i$ by $x$. \label{regionsOutsupp}
\item The cardinality of $J_1$ is of order $CN^{1-\alpha}$, where again, $C > 0$ depends only on $\mu_V$ in a neighborhood of $E$. \label{regionsCard}
\end{enumerate}
\end{lem}
\begin{proof}
The first part of statement \ref{regionsHati} holds by the observation that for $i \in J_1 \cup J_2$, $\gamma_i$ is in the bulk, so 
\begin{equation*}
0 < c \leq  \int_a^{\gamma_i} d\mu_V(x) = \frac{i}{N} \leq C < 1
\end{equation*}
for constants $C$, $c > 0$ depending only on $\mu_V$. For the second part of statement \ref{regionsHati}, the density of $\mu_V$ is bounded below uniformly in $i \in J_1\cup J_2$, so
\begin{equation*}
c |\gamma_i - \gamma_{i+1}| \leq \int_{\gamma_i}^{\gamma_{i+1}} d \mu_V(x) = \frac{1}{N}.
\end{equation*}

\noindent Statement \ref{regionsOutsupp} can be seen as follows: consider $i \in J_2$, on this set $\hat{i} \geq CN$ by \ref{regionsHati}, so uniformly in such $i$, $N^\alpha |\lambda_i - \gamma_i| \leq C N^{\alpha -1 + \xi}$, which goes to zero, while $N^\alpha|\gamma_i - E| > 2M$. On the other hand, for $i \in J_3$, we have $N^\alpha|\gamma_i -E| > \frac{1}{2}N^\alpha (E-a)\wedge (b-E)$, which goes to infinity faster than $N^\alpha|\lambda_i - \gamma_i| \leq N^{\alpha -\frac{2}{3}+ \xi}$, by our choice of $\xi$. When we substitute $\gamma_i$ by $x$, the same argument holds because $N^\alpha|x - \gamma_i| \leq N^\alpha |\gamma_i - \gamma_{i+1}|$, which is of order $N^{\alpha-1}$ on $J_2$  (as we showed in statement \ref{regionsHati}) and of order  $C N^{\alpha - \frac{2}{3}}$ on $J_3$. 

\noindent Statement \ref{regionsCard} follows by the observation that on the set $x\in [a,b]$ such that $|x - E| \leq \frac{2M}{N^\alpha}$ the density of $\mu_V$ is bounded uniformly above and below, so 
\begin{equation*}
 \frac{c}{N^\alpha} \leq \int_{|x- E| \leq \frac{2M}{N^\alpha}} d \mu_V(x) = \sum_{i \in J_1} \int_{\gamma_i}^{\gamma_{i+1}}d \mu_V(x) + O\left(\frac{1}{N}\right) \leq \frac{C}{N^\alpha},
\end{equation*}
giving the required result.
\end{proof}

\noindent The rigidity of eigenvalues, \Cref{thm::BEYrigid}, along with the previous Lemma leads to the following estimates

\begin{lem}\label{controle}
\noindent For all $0 < \xi < (1-\alpha)\wedge \frac{2}{3}$ there exists constants $C$, $c >0$ such that for $N$ large enough  we have the concentration bounds
\begin{equation}\label{controle1}
\mathbb{P}_V^N \big(|M_N(\tilde{f})| \geq CN^{\xi }  \norm{f}_{C^1(\mathbb{R})}\big) \leq e^{-N^c} \  ,  
\end{equation}
\begin{equation}\label{controle11}
\mathbb{P}_V^N \big(|M_N(\Xi^{-1}(\tilde{f})' )| \geq CN^{\alpha + \xi }  \norm{f}_{C^1(\mathbb{R})}\big) \leq e^{-N^c}  \  ,
\end{equation}
\begin{equation}\label{controle111}
\mathbb{P}_V^N \big(|M_N(\Xi^{-1}(\tilde{f})\tilde{f}' )| \geq CN^{\alpha  + \xi }  \norm{f}_{C^1(\mathbb{R})}\big) \leq e^{-N^c} \ .
\end{equation}
\end{lem}

\begin{proof}
\noindent Let $M > 1$ such that $\supp f \subset  [-M,M]$ and fix $0 < \xi < (1- \alpha)\wedge \frac{2}{3} $. For the remainder of the proof, we may assume that we are on the event $\Omega := \big\{ \forall i \ , \  |\lambda_i - \gamma_i | \leq N^{-2/3 +\xi } \  \hat{i}^{-1/3} \big\}$. This follows from the fact that, for example,
\begin{multline*}
\mathbb{P}_N^V\left(|M_N(\tilde{f})| \geq CN^\xi \|f\|_{C^1(\mathbb{R})}\right) 
\leq \mathbb{P}_N^V\left( \left\{ |M_N(\tilde{f})| \geq C N^\xi\|f\|_{C^1(\mathbb{R})}\right\} \cap \Omega\right) + \mathbb{P}_N^V(\Omega^c) ,
\end{multline*}
and by \Cref{thm::BEYrigid}, we may bound $\mathbb{P}_N^V(\Omega^c)$ by $e^{-N^c}$ for some constant $c > 0$, and $N$ large enough. On $\Omega$, the $\lambda_i$ satisfy the conditions of \Cref{regions}, we will utilize the sets $J_1$, $J_2$, and $J_3$ as defined there.

\noindent We begin by controlling \eqref{controle1}. We have that
\begin{multline}
\label{controle1triang}
|M_N(\tilde{f})| =  \left|\sum_{i=1}^N f(N^\alpha(\lambda_i - E)) - N \mu_V(\tilde f) \right|\\
		 \leq \left|\sum_{i=1}^N f(N^\alpha(\lambda_i - E)) - \sum_{i=1}^N f(N^\alpha(\gamma_i - E)) \right|
			+ \left|\sum_{i=1}^N f(N^\alpha(\gamma_i - E)) - N \mu_V(\tilde f)\right|,
\end{multline}
the first term in \eqref{controle1triang} may be bounded (on $\Omega$) by
\begin{multline*}
\left| \sum_{i=1}^N f(N^\alpha (\lambda_i - E)) - \sum_{i=1}^N f(N^\alpha (\gamma_i - E)) \right| \\
= \left| \sum_{i=1}^N N^\alpha(\lambda_i - \gamma_i) \int_0^1 f'(tN^\alpha(\lambda_i - \gamma_i) + N^\alpha(\gamma_i -E)) dt \right|\\
\leq  \sum_{i=1}^N N^{\alpha - 2/3 + \xi}\, \hat{i}^{-1/3} \int_0^1 |f'(tN^\alpha(\lambda_i - \gamma_i) + N^\alpha(\gamma_i -E))| dt ,
\end{multline*}
By \Cref{regions} \ref{regionsOutsupp}, for $N$ large enough, we have
\begin{multline}
\int_0^1 \sum_{i=1}^N N^{\alpha - 2/3 + \xi}\, \hat{i}^{-1/3}  |f'(tN^\alpha(\lambda_i - \gamma_i) + N^\alpha(\gamma_i -E))| dt \\
= \int_0^1 \sum_{i \in J_1 } N^{\alpha - 2/3 + \xi}\, \hat{i}^{-1/3}  |f'(tN^\alpha(\lambda_i - \gamma_i) + N^\alpha(\gamma_i -E))| dt \\
\leq  \sum_{i\in J_1 } N^{\alpha - 1 + \xi}\|f\|_{C^1(\mathbb{R})} \leq C N^{\xi} \|f\|_{C^1(\mathbb{R})},
\end{multline}
where, in the third line we used \Cref{regions} \ref{regionsHati} and \ref{regionsCard} in order. Thus 
\begin{equation*}
\begin{split}
\left| \sum_{i=1}^N f(N^\alpha (\lambda_i - E)) - \sum_{i=1}^N f(N^\alpha (\gamma_i - E)) \right| \leq C  N^{\xi} \norm{f}_{C^1(\mathbb{R})}.
\end{split}
\end{equation*}

\noindent For the second term in \eqref{controle1triang}, 
\begin{multline*}
\left| \sum_{i=1}^N f(N^\alpha (\gamma_i - E)) -  N  \int_a^b f(N^\alpha(x-E)) d\mu_V(x)\right| \\
\leq  N  \sum_{i\in J_1} \int_{\gamma_i}^{\gamma_{i+1}} \big| f(N^\alpha (\gamma_i - E)) - f(N^\alpha(x-E))    \big| d\mu_V(x) \\
\leq  N^{1+\alpha} \norm{f}_{C^1(\mathbb{R})}  |J_1| \sup_{i \in J_1} (\gamma_{i+1} - \gamma_i)\int_{\gamma_i}^{\gamma_{i+1}}  d\mu_V(x) \leq  C  \norm{f}_{C^1(\mathbb{R})} 
\end{multline*}
since the spacing  of the quantiles in $J_1$ is bounded by $\frac{C}{N}$. This proves \eqref{controle1}. 

\noindent We now proceed with the proof of \eqref{controle11}. 
\begin{equation*}
\begin{split}
 \Big |M_N(\Xi^{-1}(\tilde{f})' )\Big|  &= \Big| \sum_{i=1}^N \Big( \Xi^{-1}(\tilde{f})' (\lambda_i) - N  \int_{\gamma_i}^{\gamma_{i+1}} \Xi^{-1}(\tilde{f})' (x) d\mu_V(x)  \Big) \Big|  \\
 &\leq  N   \sum_{i = 1}^N \int_{\gamma_i}^{\gamma_{i+1}} \Big| \Xi^{-1}(\tilde{f})' (\lambda_i) -  \Xi^{-1}(\tilde{f})' (x) \Big| d\mu_V(x)\\ 
 &\leq  N \sum_{i=1}^N \int_{\gamma_i}^{\gamma_{i+1}} \int_0^1 |\lambda_i - x| \left|\Xi^{-1}(\tilde{f})^{(2)} (t (\lambda_i - x) + x )\right|\, dt \, d\mu_V(x) ,\\
 \end{split}
\end{equation*}
Recall from the proof of Lemma \ref{regions} that uniformly in $i \in J_2$ and $x\in [\gamma_i,\gamma_{i+1}]$, $|\gamma_i - E| \geq \frac{2M}{N^\alpha}$ while $|x - \gamma_i| \leq \frac{C}{N}$; further, $|\lambda_i - x| \leq C N^{-1 + \xi}$ so for $N$ large enough we can replace $|t(\lambda_i - x) + (\gamma_i - E)|$ by $|\gamma_i - E|$ uniformly in $t\in [0;1]$. Likewise, uniformly in $i \in J_3$ and $x\in [\gamma_i, \gamma_{i+1}]$,  $| \gamma_i - E| \geq C$ while $|x - \gamma_i| \leq CN^{-\frac{2}{3}}$; further $|\lambda_i - x| \leq CN^{-1 + \xi}$ so for $N$ large enough we can replace $|t(\lambda_i - x) + (\gamma_i - E)|$ by a constant $C$ uniformly in $t\in [0;1]$ for what follows.

\noindent For $i \in J_2$, by the observations in the previous paragraph, along with \Cref{regions} \ref{regionsOutsupp}, \Cref{bornebounds} \cref{borneexsup},  and \Cref{regions} \ref{regionsHati},
\begin{multline*}
N \sum_{i\in J_2} \int_{\gamma_i}^{\gamma_{i+1}}\int_0^1 |\lambda_i - x| \left|\Xi^{-1}(\tilde{f})^{(2)} (t (\lambda_i - x) + x )\right|\, dt d\mu_V(x) \\
\leq N \sum_{i \in J_2}\int_{\gamma_i}^{\gamma_{i+1}} \int_0^1 \frac{C |\lambda_i - x| }{N^\alpha (| t(\lambda_i - x) + x - E|^3 \wedge 1)} dt d\mu_V(x) \leq \sum_{i \in J_2}   \frac{C N^{\xi -1 - \alpha} }{(\gamma_i - E)^3} ,
\end{multline*}

\noindent The same reasoning for $i \in J_3$ yields
\begin{equation*}
N \sum_{i\in J_3} \int_{\gamma_i}^{\gamma_{i+1}}\int_0^1 |\lambda_i - x| \left|\Xi^{-1}(\tilde{f})^{(2)} (t (\lambda_i - x) + x )\right|\, dt \, d\mu_V(x) \leq \sum_{i \in J_3} C N^{\xi - \alpha - \frac{2}{3}}\, \hat{i}^{-\frac{1}{3}}.
\end{equation*}

\noindent For $i \in J_1$, by \Cref{bornebounds} \cref{bornesup} and \Cref{regions} \ref{regionsHati},
\begin{multline*}
N \sum_{i\in J_1} \int_{\gamma_i}^{\gamma_{i+1}}\int_0^1 |\lambda_i - x| \left|\Xi^{-1}(\tilde{f})^{(2)} (t (\lambda_i - x) + x )\right|\, dt \, d\mu_V(x) \\
\leq N \sum_{i \in J_1} \int_{\gamma_i}^{\gamma_{i+1}} C N^{2\alpha} \log N |\lambda_i - x| d\mu_V(x)  \leq \sum_{i \in J_1} C N^{2\alpha + \xi - 1} \log N.
\end{multline*}
It follows that
\begin{equation*}
\begin{split}
 \left|M_N(\Xi^{-1}(\tilde{f})' )\right|  & \leq  \sum_{i\in J_1} C N^{2\alpha + \xi - 1} \log N + \sum_{i \in J_2} \frac{CN^{\xi -1 -\alpha}}{(\gamma_i - E)^3} + \sum_{i \in J_3} C N^{\xi - \alpha - \frac{2}{3}} \hat{i}^{\frac{1}{3}}\\
 &\leq  C N^{\alpha + \xi} \log N  +  C N^{\xi + \alpha} \leq CN^{\alpha + \xi} \log N,
\end{split}
\end{equation*}
where we have used $|J_1| \leq C N^{1-\alpha}$ from \Cref{regions}, and the following estimates:
\begin{align*}
\sum_{i \in J_2}\frac{N^{\xi - \alpha- 1}}{(\gamma_i - E)^3} &\leq CN^{\xi - \alpha} \left( \int_{a}^{E - \frac{2M}{N^\alpha}}\frac{dx}{(x - E)^3} + \int_{E+\frac{2M}{N^\alpha}}^b \frac{dx}{(x-E)^3}\right) \leq C N^{\xi + \alpha},\\
C N^{\xi - \alpha - \frac{2}{3}} \sum_{i \in J_3} \hat{i}^{-\frac{1}{3}} &\leq  C N^{\xi - \alpha} \times \frac{1}{N}\sum_{i = 1}^N\left(\frac{i}{N}\right)^{-\frac{1}{3}} \leq C N^{\xi - \alpha}.
\end{align*}
\noindent This proves \eqref{controle11}. The bound \eqref{controle111} is obtained in a similar way and we omit the details.
\end{proof}

\noindent For convenience we introduce the following notation: for a sequence of random variable $(X_N)_{N\in \mathbb{N}}$ we write  $X_N = \omega(1)$ if there exists  constants $c$,  $C$ and $\delta > 0$ such that the bound $ | X_N | \leq \frac{C}{N^\delta}$ holds with probability greater  than $1 - e^{-N^c}$.

\noindent Using  Lemma \ref{controle} we  prove the following bounds:

\begin{lem}\label{controle2}
The  following estimates hold:
\begin{equation}\label{mean}
L_N \Big( \Xi^{-1}(\tilde{f})' \Big) = \omega(1) \ , 
\end{equation}

\begin{equation}\label{variance}
L_N \Big(\Xi^{-1}(\tilde{f}) \tilde{f}'  \Big) + \sigma_f^2 =  \omega(1)  \ , 
\end{equation}

\begin{equation}\label{cross}
\frac{1}{N }   \iint  \frac{\Xi^{-1}(\tilde{f})(x)-\Xi^{-1}(\tilde{f})(y)}{x-y} dM_N(x) dM_N(y)    = \omega(1) \ . 
\end{equation}
\end{lem}
\begin{proof}
For both \eqref{mean} and \eqref{variance}, we use 
\begin{align*}
L_N \Big( \Xi^{-1}(\tilde{f})' \Big) &= \frac{M_N(\Xi^{-1}(\tilde{f})')}{N} + \mu_V(\Xi^{-1}(\tilde{f})'),\\
L_N \left(\Xi^{-1}(\tilde{f}) \tilde{f}'  \right) &=  \frac{ M_N \left(\Xi^{-1}(\tilde{f}) \tilde{f}'  \right)}{N} + \mu_V\left( \Xi^{-1}(\tilde{f}) \tilde{f}' \right),
\end{align*}
\Cref{controle} implies that the first term in both equations are $\omega(1)$ so \eqref{mean} and \eqref{variance} simplify to deterministic statements about the speed of convergence of the integrals against $\mu_V$ above. 

\noindent To show \eqref{mean}, integration by parts yields: 
\begin{equation*}
\int (\Xi^{-1}\tilde{f})'(x) d\mu_{V}(x) =  - \int_{a}^{b} (\Xi^{-1} \tilde{f})(x) (S'(x) \sigma(x) + S(x) \sigma'(x) )dx,
\end{equation*}
inserting the formula for $\Xi^{-1} \tilde{f}$ we obtain
\begin{equation*}
\left|\int (\Xi^{-1}\tilde{f})'(x) d\mu_{V}(x)\right| \leq \frac{1}{\beta \pi^2} \int_{a}^{b}\int_{a}^{b}\left| \frac{\tilde{f}(x) - \tilde{f}(y)}{y-x}\right| \left(\left| \frac{S'(x) \sigma(x)}{S(x) \sigma(y)}\right| + \left|\frac{\sigma'(x)}{\sigma(y)}\right| \right)dxdy \ .
\end{equation*}
Recall that $S$ is bounded below on $[a,b]$, $S'$ is bounded above on $[a,b]$, further, up to a constant, $\frac{\sigma'(x)}{\sigma(y)}$ can be bounded above by $(\sigma(x)\sigma(y))^{-1}$. We define the sets
\begin{align*}
A_N &:= [ N^\alpha(a-E) ; N^\alpha(b-E)] ,\\
B_N &:= \left[\frac{1}{2} N^\alpha(a-E) ;\frac{1}{2} N^\alpha(b-E)\right].
\end{align*}

\noindent By the observations above, and the change  of variable $u= N^\alpha (x-E)$ and  $v= N^\alpha (y-E)$ we get 
\begin{multline}
\label{meanupbound}
\left|\int (\Xi^{-1}\tilde{f})'(x) d\mu_{V}(x)\right| \\
\leq \frac{C}{N^\alpha} \iint_{A_N^2}\left| \frac{{f}(u) - {f}(v)}{u-v}\right| \left(\frac{ \sigma(E+\frac{u}{N^\alpha})}{ \sigma(E+ \frac{v}{N^\alpha})} + \frac{1}{\sigma(E+ \frac{u}{N^\alpha}) \sigma(E+\frac{v}{N^\alpha})}\right)dudv.
\end{multline}
For large enough $N$, on the set $(u,v) \in (A_N\backslash B_N)^2$, the function $|f(u) - f(v)|$ is always zero, thus the integral on the right above can be divided into integrals over the sets:
\begin{equation}
\label{meansplit}
(A_N \times A_N) \cap (A_N \backslash B_N \times A_N \backslash B_N)^c = B_N\times B_N \cup B_N \times (A_N \backslash B_N) \cup (A_N \backslash B_N) \times B_N.
\end{equation}

\noindent  We bound the integral in \eqref{meanupbound} over each set in \eqref{meansplit}. We begin with the first set in \eqref{meansplit}. For $(u,v) \in B_N\times B_N$, $\sigma(E + \frac{u}{N^\alpha})$ and $\sigma(E + \frac{v}{N^\alpha})$ are uniformly bounded above and below. Therefore, the integral in \eqref{meanupbound} can be bounded in this region by 
\begin{multline*}
\iint_{B_N^2} \left|\frac{f(u) - f(v)}{u-v} \right|du dv   \\
=\iint_{[-M;M]^2}\left|\frac{f(u) - f(v)}{u - v} \right| du dv + 2 \int_{-M}^M \int_{ B_N \cap\{|u| \geq M\}} \left|\frac{f(v)}{u-v}\right| \,du \,dv,
\end{multline*}
the integral over $[-M;M]^2$ exists by the differentiability of $f$, while:
\begin{multline*}
\int_{-M}^M \int_{ B_N \cap\{|u| \geq M\}} \left|\frac{f(v)}{u-v}\right| \,du \,dv   \leq C\int_{-M}^M |f(v)| \log [N|v + M||v - M|] \,dv \leq C\log N,
\end{multline*}
for $N$ large enough.

\noindent For the second set in \eqref{meanupbound}, observe that for $(u,v)\in B_N \times (A_N \backslash B_N)$, $f(v)$ is 0 for $N$ sufficiently large, and $\sigma(E + \frac{u}{N^\alpha})$ is bounded uniformly above and below while $f(u)$ is 0 outside $[-M;M]$. This implies that the integral in \eqref{meanupbound} can be bounded in this region by
\begin{multline*}
 \int_{A_N \backslash B_N } \int_{-M}^M \left| \frac{f(u)}{u-v}\right| \left( \frac{ \sigma(E+\frac{u}{N^\alpha})}{ \sigma(E+ \frac{v}{N^\alpha})} + \frac{1}{\sigma(E+ \frac{u}{N^\alpha}) \sigma(E+\frac{v}{N^\alpha})} \right)dudv \\ 
\leq \frac{C\|f\|_{C(\mathbb{R})}}{N^\alpha}   \int_{A_N \backslash B_N } \frac{1}{ \sigma(E+ \frac{v}{N^\alpha})} dv  \leq C \ , 
\end{multline*}
where in the final line we used $|u-v| \geq c N^\alpha$ for $u \in [-M ; M]$ , $v \in A_N \backslash B_N$. 

\noindent We can do similarly for the third set in \eqref{meanupbound} and putting together these bounds on the right hand side of \eqref{meanupbound} gives
\begin{equation*}
\left|\int (\Xi^{-1}\tilde{f})'(x) d\mu_{V}(x)\right| \leq \frac{C\log N}{N^\alpha},
\end{equation*}
which is $\omega(1)$ as claimed.

\noindent We continue with \eqref{variance}. Recall that we reduced this problem to computing the limit of $\mu_V ( \Xi^{-1}(\tilde{f}) \tilde{f}')$. Using the inversion formula we see that 
\begin{equation*}
\int \Xi^{-1} \tilde{f}(x) \tilde{f}'(x) d\mu_{V}(x) = - \frac{1}{\beta \pi^2}\int_a^b \int_a^b \frac{\sigma(x)\tilde{f}'(x)(\tilde{f}(x) - \tilde{f}(y))}{\sigma(y)(x - y)}dxdy  
\end{equation*}

\noindent Observe that 
\begin{align*}
\frac{1}{2}\partial_x (\tilde{f}(x) - \tilde{f}(y))^2 &= \tilde{f}'(x) ( \tilde{f}(x) - \tilde{f}(y)), \\
\partial_x \left(\frac{\sigma(x)}{x-y}\right) &= \frac{-\frac{1}{2}(a + b)(x + y) + ab + xy}{\sigma(x)(x-y)^2}. 
\end{align*}
Therefore, integration by parts yields
\begin{multline*}
\int \Xi^{-1} \tilde{f}(x) \tilde{f}'(x) d\mu_{V}(x)  =  - \frac{1}{2\beta \pi^2}\int_a^b \int_a^b \frac{\sigma(x)\partial_x(\tilde{f}(x) - \tilde{f}(y))^2}{\sigma(y)(x - y)}dxdy  \\
=  \frac{1}{ 2 \beta \pi^2}\int_a^b \int_a^b \left( \frac{\tilde{f}(x) - \tilde{f}(y)}{x - y} \right)^2 \left( \frac{ab+xy - \frac{1}{2}(a+b)(x+y)}{\sigma(x) \sigma(y)} \right) dxdy , 
\end{multline*}
By changing variables again to $(u,v) = (N^\alpha(x-E),N^\alpha(y-E))$  and observing that
\begin{equation*}
ab+xy - \frac{1}{2}(a+b)(x+y) = - \sigma(E)^2 + \frac{u + v}{N^\alpha}\left(\frac{a + b}{2} + E\right)  + \frac{uv}{N^{2\alpha}},
\end{equation*}
we obtain
\begin{multline}
\label{varprelim}
\int \Xi^{-1} \tilde{f}(x) \tilde{f}'(x) d\mu_{V}(x)  \\
=   - \frac{1}{ 2 \beta \pi^2}\iint_{A_N^2} \left( \frac{{f}(u) - {f}(v)}{u - v} \right)^2 \left( \frac{\sigma(E)^2 - \frac{u+v}{N^\alpha}(\frac{a+b}{2}+E) - \frac{uv}{N^{2\alpha}}}{\sigma(E +\frac{u}{N^\alpha})  \sigma(E+ \frac{v}{N^\alpha})} \right) dudv.
\end{multline}
As before, $(f(u) - f(v))^2$ is zero for all $(u,v) \in (A_N\backslash B_N)^2$ for large enough $N$, therefore  we split the above integral into the regions defined in \eqref{meansplit}.

\noindent Notice that uniformly in $u \in B_N$ 
\begin{equation*}
\frac{1}{\sigma(E +\frac{u}{N^\alpha})} = \frac{1}{\sigma(E)} + O\left(\frac{|u|}{N^\alpha}\right) \ ,
\end{equation*}
and further notice $(u+v)/N^\alpha$ and $uv/N^{2\alpha}$ are bounded uniformly by constants in the entire region $A_N\times A_N$ and converge pointwise to 0 for each $(u,v)$. 

\noindent Consequently the integral \eqref{varprelim} over the region $B_N\times B_N$ is:
\begin{multline}
\label{varBN2}
 \iint_{B_N^2} \left( \frac{{f}(u) - {f}(v)}{u - v} \right)^2 \left( \frac{\sigma(E)^2 - \frac{u+v}{N^\alpha}\left(\frac{a + b}{2} + E\right)  - \frac{uv}{N^{2\alpha}}}{\sigma(E +\frac{u}{N^\alpha})  \sigma(E+ \frac{v}{N^\alpha})} \right) dudv \\
  =  \iint_{B_N^2} \left( \frac{{f}(u) - {f}(v)}{u - v} \right)^2\left( 1 - \frac{u+v}{N^\alpha\sigma(E)^2}\left(\frac{a+b}{2} + E\right) - \frac{uv}{N^{2\alpha} \sigma(E)^2}\right) du dv  \\
  + O \left( \frac{1}{N^\alpha}  \iint_{B_N^2} \left( \frac{{f}(u) - {f}(v)}{u - v} \right)^2 (|u|  + |v|) du dv \right) \ ,
\end{multline}
the first term of \eqref{varBN2} is equal to,
\begin{equation*}
 \frac{1}{2\beta \pi^2} \iint \left(\frac{f(u) - f(v)}{u-v}\right)^2 du dv +  O\Big(\frac{1}{N^\alpha}\Big)
\end{equation*}
while the second term in \eqref{varBN2} can be written as
\begin{multline*}
 \iint_{B_N^2} \left( \frac{{f}(u) - {f}(v)}{u - v} \right)^2 (|u|  + |v|) du dv =  \iint_{[-M;M]^2} \left( \frac{{f}(u) - {f}(v)}{u - v} \right)^2 (|u|  + |v|) du dv \\
 + 2\int_{-M}^M \int_{B_N\cap \{|u| \geq M\}}\left( \frac{{f}(v)}{u - v} \right)^2 (|u|  + |v|) du dv,
\end{multline*}
the integral over $[-M;M]^2$ is finite by differentiability of $f$ while the second is bounded by
\begin{multline*}
\int_{-M}^M \int_{B_N \cap \{|u| \geq M\}} |f(v)|^2\left( \frac{1}{|u-v|} + \frac{2|v|}{|u-v|^2}\right) du dv \\
\leq  C \int_{-M}^M |f(v)|^2 \left(\frac{1}{|v - M|} + \frac{1}{|M+v|} + \log[N|v-M||v+M|]\right) \leq C \log N
\end{multline*}
since $\supp f \subset [-M,M]$.

\noindent In the region $(u,v) \in  B_N \times (A_N\backslash B_N)$, $\sigma( E+ \frac{u}{N^\alpha})$ is bounded above and below while, for $N$ large enough $f(v) = 0$, thus the integral over $B_N \times (A_N\backslash B_N)$ is bounded above by
\begin{multline*}
 \int_{A_N \backslash B_N} \int_{B_N} \left( \frac{{f}(u) - {f}(v)}{u - v} \right)^2 \left( \frac{1}{\sigma(E +\frac{u}{N^\alpha})  \sigma(E+ \frac{v}{N^\alpha})} \right) dudv  \\
\leq   \int_{A_N \backslash B_N} \int_{-M}^M \left( \frac{{f}(u)}{u - v} \right)^2 \frac{1}{\sigma(E+ \frac{v}{N^\alpha})} dudv 
  \leq   \frac{C}{N^{2\alpha}}  \int_{A_N \backslash B_N} \frac{1}{ \sigma(E+ \frac{v}{N^\alpha})} dv \leq \frac{C}{N^\alpha} \ ,
\end{multline*}
where in the second line we used $|u - v| \geq cN^\alpha$ for $u\in [-M;M]$ and $v\in A_N\backslash B$. By symmetry of the integrand in \eqref{varprelim} this argument extends to the region $(u,v) \in (A_N \backslash B_N) \times B_N$. 

\noindent Altogether, our bounds show 
\begin{equation*}
\int \Xi^{-1} \tilde{f}(x) \tilde{f}'(x) d\mu_{V}(x)  = - \frac{1}{2 \beta \pi^2} \iint \left(\frac{f(x) - f(y)}{x-y}\right)^2 dx dy + O\left( \frac{\log N}{N^\alpha}\right),
\end{equation*}
which shows \eqref{variance}.

\noindent We conclude by  proving \eqref{cross}. The proof will be similar to the proof of \Cref{controle}.  As in \Cref{controle} we may restrict our attention to the event $\Omega = \{\forall i: |\lambda_i - \gamma_i | \leq N^{-\frac{2}{3} + \xi} \hat{i}^{-\frac{1}{3}}\}$ by applying \Cref{thm::BEYrigid}.  Further, we use again the sets $J_1$, $J_2$, and $J_3$ defined in \Cref{regions}.

\noindent Define for $j\in \{1,2,3\}$: 
\begin{equation*}
M_N^{(j)} = \sum_{i \in J_j} \big( \delta_{\lambda_i} - N \mathbbm{1}_{[{\gamma_i},{\gamma_{i+1}}]} \mu_V\big) \, 
\end{equation*}
so that $M_N = M_N^{(1)} + M_N^{(2)} + M_N^{(3)}$. We can  write
\begin{multline*}
  \iint  \frac{\Xi^{-1}(\tilde{f})(x)-\Xi^{-1}(\tilde{f})(y)}{x-y} dM_N(x) dM_N(y)    \\
  = \sum_{1 \leq j_1 ,  j_2 \leq 3}   \iint  \frac{\Xi^{-1}(\tilde{f})(x)-\Xi^{-1}(\tilde{f})(y)}{x-y} dM_N^{(j_1)} (x) dM_N^{(j_2)} (y)
\end{multline*}
\noindent Integrating repeatedly for each $(j_1,j_2)$ yields: 
\begin{multline}\label{crosscont}
  \iint  \frac{\Xi^{-1}(\tilde{f})(x)-\Xi^{-1}(\tilde{f})(y)}{x-y} dM_N^{(j_1)} (x) dM_N^{(j_2)} (y) = \\
 N^2 \sum_{\substack{ i_1 \in J_{j_1} \\  i_2 \in J_{j_2}  }} 
 \int_{\gamma_{i_1}}^{\gamma_{i_1+1}} d\mu_V(x_1)
 \int_{\gamma_{i_2}}^{\gamma_{i_2+1}} d\mu_V(x_2)
 \int_T \,du \,dv \,dt  \Big\{(\lambda_{i_1} -x_1) (\lambda_{i_2} -x_2) t (1-t)\\
  \times \Xi^{-1}(\tilde{f})^{(3)} \big( tv(\lambda_{i_1} - x_1) + ut(x_2 - \lambda_{i_2}) + u(\lambda_{i_2} - x_2) + t(x_1 - x_2)+ x_2\big)\Big\} \\
\end{multline}
where $T = [0;1]^3$. We will bound \eqref{crosscont} for each pair  $(j_1,j_2)$.

\noindent \textbf{For $(j_1,j_2) = (1,1)$.} Recall by \Cref{regions} \ref{regionsCard} that $|J_1|\leq C N^{1-\alpha}$, and further from the proof of \Cref{regions} uniformly in $i\in J_1$, $|\lambda_i - x| \leq CN^{\xi-1}$ whenever $x\in [\gamma_i,\gamma_{i+1}]$. We use \eqref{crosscont}, \Cref{bornebounds} \cref{bornesup} to obtain the upper bound
\begin{multline*}
\iint  \frac{\Xi^{-1}(\tilde{f})(x)-\Xi^{-1}(\tilde{f})(y)}{x-y} dM_N^{(1)} (x) dM_N^{(1)} (y) \leq \\
 N^2 \sum_{\substack{i_1\in J_1 \\ i_2 \in J_1}}\int_{\gamma_{i_1}}^{\gamma_{i_1 + 1}}\int_{\gamma_{i_2}}^{\gamma_{i_2+1}} N^{3\alpha} \log N |\lambda_{i_1}  - x_1| |\lambda_{i_2} - x_2|\, d\mu_V(x_1) \, d\mu_V(x_2)
\leq C N^{2 \xi + \alpha} \log N,
\end{multline*}
which is $\omega(1)$ when divided by $N$.

\noindent \textbf{For $(j_1, j_2) = (2,2)$.} We remark that the strategy is not as straightforward as the case $i \in J_2$ in the proof of \Cref{controle} \cref{controle11}, this is because the term $t(x_1 - x_2) + x_2$ appearing as an argument in \eqref{crosscont} may enter a neighborhood of 0 depending on the indices $i_1, i_2 \in J_2$; so we may not use the bound \Cref{bornebounds} \cref{borneexsup} uniformly in $i_1$, $i_2 \in J_2$. Some care is needed also because $M_N$ is a signed measure so $|M_N(g)|$ need not be bounded by $M_N(|g|)$.

\noindent It will be convenient to use directly \cref{invint} from the proof of \Cref{bornebounds} (this can be done as $J_2$ is located outside the support of $f$). We can write
\begin{multline}
\label{MN2split}
\frac{\Xi^{-1}(\tilde{f})(x)-\Xi^{-1}(\tilde{f})(y)}{x-y} \\
= \frac{1}{\beta \pi^2}\int_{-M}^M \frac{f(u)}{\sigma(E + \frac{u}{N^\alpha})(x-y)} \left(\frac{1}{S(y)(u - N^\alpha(y-E))} - \frac{1}{S(x)(u - N^\alpha(x-E))} \right)\,du\\
= \frac{1}{\beta \pi^2}\int_{-M}^M \frac{f(u)}{\sigma(E + \frac{u}{N^\alpha})} \Bigg\{\frac{S(x) - S(y)}{(x-y)}\frac{1}{S(x)S(y)(u - N^\alpha(y-E))} \\
+ \frac{N^\alpha}{S(x)(u - N^\alpha(x-E))(u - N^\alpha (y-E))}\Bigg\}\,du.
\end{multline}
When we integrate the term on the third line of \eqref{MN2split} against $M_N^{(2)}\otimes M_N^{(2)}$, we obtain
\begin{equation}
\label{MN2first}
\int_{-M}^M \frac{f(u)}{\sigma(E + \frac{u}{N^\alpha})}\left\{\int M_N^{(2)}\left(\int_0^1\frac{ S'(t(\cdot - y) + y)}{S(\cdot)S(y)} dt \right) \frac{1}{(u - N^\alpha(y-E))}dM_N^{(2)}(y)\right\}du,
\end{equation}
define the function
\begin{equation*}
g(y) := M_N^{(2)}\left(\int_0^1\frac{ S'(t(\cdot - y) + y)}{S(\cdot)S(y)} dt \right), 
\end{equation*}
first, $g(y)$ is bounded for any $y\in [a; b]$: 
\begin{multline*}
\left|M_N^{(2)}\left(\int_0^1\frac{ S'(t(\cdot - y) + y)}{S(\cdot)S(y)} dt \right)\right| \\
=\left|\frac{N}{S(y)}\sum_{i\in J_2} \int_{\gamma_i}^{\gamma_{i+1}}\int_0^1   \left(\frac{S'(t(\lambda_i - y) + y)}{S(\lambda_i)} - \frac{S'(t(x - y) + y)}{S(x)} \right)\,dt \,d\mu_V(x)\right|\\
\leq \left|\frac{N}{S(y)}\sum_{i\in J_2} \int_{\gamma_i}^{\gamma_{i+1}}\int_0^1   \frac{S'(t(\lambda_i - y) + y) - S'(t(x - y) + y)}{S(\lambda_i)} dt d \mu_V(x)\right| \\
+\left| \frac{N}{S(y)}\sum_{i\in J_2} \int_{\gamma_i}^{\gamma_{i+1}}\int_0^1   \frac{S(x) - S(\lambda_i)}{S(x)S(\lambda_i)} S'(t(x-y) + y) \,dt \,d\mu_V(x)\right| \leq CN^\xi,
\end{multline*}
where in the final line we used $S$ and $S'$ are smooth on $[a;b]$ (and therefore uniformly Lipschitz), $S > 0$ in a neighborhood of $[a;b]$,  further $|x - \lambda_i| \leq C N^{\xi-1}$, and $|J_2| \leq CN$. Moreover, $g(y)$ is uniformly Lipschitz in $[a; b]$ with constant $CN^\xi$, since:
\begin{multline*}
M_N^{(2)}\left(\int_0^1\frac{S'(t(\cdot - y) + y)}{S(\cdot)S(y)} - \frac{S'(t(\cdot -z) + z)}{S(\cdot)S(z)} dt \right) = \\
(z-y) M_N^{(2)}\left(\int_0^1\int_0^1\frac{t S''(ut(z - y)+ t(\cdot - z) + y)}{S(\cdot)S(y)} dt du\right) \\
+ \frac{S(z) - S(y)}{S(z)S(y)}M_N^{(2)}\left(\int_0^1\frac{ S'(t(\cdot -z) + z)}{S(\cdot)} dt \right)
\end{multline*}
and both terms appearing in $M_N^{(2)}$ above are of the same form as $g$ so they are bounded by $CN^\xi$. Returning to \eqref{MN2first}, we may bound
\begin{multline*}
\left| M_N^{(2)}\left( \frac{g(y)}{u - N^\alpha(y-E)}\right)\right| \\
=\left|N\sum_{i\in J_2} \int_{\gamma_i}^{\gamma_{i+1}} \frac{g(\lambda_i)- g(x)}{(u - N^\alpha (\lambda_i-E))} + \frac{N^\alpha(\lambda_i - x) g(x)}{(u - N^\alpha(\lambda_i - E))(u - N^\alpha(x-E))}\,d\mu_V(x)\right|\\
\leq \int_{[a;b]\cap\{|x - E|\geq \frac{2M }{N^\alpha}\}} \frac{C N^{2\xi}}{|u - N^\alpha(x-E)|}  + \frac{ CN^{2\xi + \alpha} }{(u - N^\alpha(x-E))^2} \,dx\\
\leq C N^{2\xi - \alpha} \log N + C N^{2\xi},
\end{multline*}
uniformly in $u$. Thus \eqref{MN2first} is bounded by $CN^{2\xi}$ as $f$ is bounded.

\noindent The remaining term in \eqref{MN2split} is 
\begin{multline}
\label{MN2scnd}
 N^\alpha \int_{-M}^M \frac{f(u)}{\sigma(E + \frac{u}{N^\alpha})} M_N^{(2)}\left( \frac{1}{S(\cdot)(u - N^\alpha(\, \cdot-E))}\right)M_N^{(2)}\left( \frac{1}{u - N^\alpha(\, \cdot - E)}\right) \,du .
\end{multline}
Repeating our argument in the previous paragraph gives:
\begin{align*}
\left|M_N^{(2)}\left( \frac{1}{S(\cdot)(u - N^\alpha(\, \cdot-E))}\right)\right| &\leq C N^{\xi - \alpha} \log N + C N^\xi,\\
\left|M_N^{(2)}\left( \frac{1}{u - N^\alpha(\, \cdot - E)}\right)\right| &\leq CN^\xi,
\end{align*}
where in the first inequality we use $1/S$ is uniformly bounded and uniformly Lipschitz on $[a;b]$. Inserting the bounds into \eqref{MN2scnd} gives an upper bound of $C N^{2\xi + \alpha}$, as $f$ is bounded.

\noindent Altogether 
\begin{equation*}
\left|\frac{\Xi^{-1}(\tilde{f})(x)-\Xi^{-1}(\tilde{f})(y)}{x-y} dM_N^{(2)}(x) dM_N^{(2)}(y) \right|  \leq C N^{2\xi},
\end{equation*}
which is $\omega(1)$ when divided by $N$.

\noindent\textbf{For $(j_1,j_2) = (3,3)$.} We bound as in the previous case, except now we define
\begin{equation*}
g(y) = M_N^{(3)}\left(\int_0^1\frac{ S'(t(\cdot - y) + y)}{S(\cdot)S(y)} dt \right),   
\end{equation*}
and apply (for $x$ in the region defined in $J_3$):
\begin{equation*}
|\lambda_i - x|  \leq N^{-\frac{2}{3} +\xi } \hat{i}^{-\frac{1}{3}} , \quad \left| \frac{1}{ (u - N^\alpha(x -E))} \right|  \leq \frac{C}{N^\alpha},
\end{equation*}
to obtain
\begin{align*}
\left| M_N^{(3)}\left( \frac{g(y)}{u - N^\alpha(y-E)}\right)\right| &\leq C N^{2\xi - \alpha},\\
\left| M_N^{(3)}\left( \frac{1}{S(\cdot)(u - N^\alpha(\, \cdot-E))}\right)\right| &\leq C N^{\xi - \alpha},\\
\left| M_N^{(3)}\left( \frac{1}{u - N^\alpha(\, \cdot - E)}\right)\right| &\leq CN^{\xi - \alpha},
\end{align*}
altogether giving
\begin{equation*}
\left|\iint\frac{\Xi^{-1}(\tilde{f})(x)-\Xi^{-1}(\tilde{f})(y)}{x-y} dM_N^{(3)}(x) dM_N^{(3)}(y) \right|  \leq C N^{2\xi - \alpha},
\end{equation*}
which is $\omega(1)$ when divided by $N$.

\noindent \textbf{For $(j_1,j_2) =  (1,2)$.} By the bounds $|\lambda_{i_j} - \gamma_{i_j}| \leq C N^{\xi-1}$,  $|\gamma_{i_j} - x_j | \leq \frac{C}{N}$ for $x_j \in [\gamma_{i_j} ; \gamma_{i_j+1}]$, whenever
\begin{equation}
\label{crossTermOutSupp}
N^\alpha |tv (\lambda_{i_1} - x_1)  + ut(x_2 - \lambda_{i_2}) + t(x_{1} - x_2) + u (\lambda_{i_2} - x_2)   + x_{2}  - E| \geq M + 1,
\end{equation}
we have
\begin{equation*}
|t(\gamma_{j_1} - \gamma_{j_2}) + (\gamma_{j_2} - E)| + CN^{\xi - 1} \geq M + 1,
\end{equation*}
by triangle inequality. It follows that for $N$ sufficiently large, uniformly in $x_1 \in [\gamma_{i_1}; \gamma_{i_1 + 1}]$, $x_2 \in [\gamma_{i_2};\gamma_{i_2 + 1}]$, $u,v \in [0;1]$
\begin{multline*}
\frac{1}{|tv (\lambda_{i_1} - x_1)  + ut(x_2 - \lambda_{i_2}) + t(x_{1} - x_2) + u (\lambda_{i_2} - x_2)   + x_{2}  - E|}  \\
\leq \frac{C}{|t(\gamma_{i_1} - \gamma_{i_2}) + (\gamma_{i_2} - E)|},
\end{multline*}
where the constant $C$ only depends on $M$. Therefore, whenever \eqref{crossTermOutSupp} is satisfied, applying \Cref{bornebounds} \cref{borneexsup} yields
\begin{multline}
\label{simplerexsup}
\left|  \Xi^{-1}(\tilde{f})^{(3)} \big(tv (\lambda_{i_1} - x_1)  + ut(x_2 - \lambda_{i_2}) + t(x_{1} - x_2) + u (\lambda_{i_2} - x_2)   + x_{2} \big) \right|  \\
\leq  \frac{C}{N^\alpha \left( (t (\gamma_{i_1} - \gamma_{i_2}) + \gamma_{i_2}  - E)^4 \wedge 1   \right)}.
\end{multline}

\noindent Now, let $ t \in (0 , 1) $ fixed and define the sets
\begin{equation*}
\begin{split}
K_t^1 &:= \left\{ j\in J_2 \ , \ t \left( E - \frac{2M}{N^\alpha}  - \gamma_j\right) + \gamma_j - E \geq \frac{2M}{N^\alpha} \right\}, \\
K_t^2 &:= \left\{ j\in J_2 \ , \ t \left( E + \frac{2M}{N^\alpha}  - \gamma_j\right) + \gamma_j - E \leq - \frac{2M}{N^\alpha} \right\},  \\
K_t &:= K_t^1 \cup K_t^2.
\end{split}
\end{equation*} 
By construction, if $i_2 \in K_t^1$ then 
\begin{equation*}
| t (\gamma_{i_1} - \gamma_{i_2}) + (\gamma_{i_2} - E)| \geq \frac{2M}{N^\alpha}
\end{equation*}
uniformly in $i_1 \in J_1$.  Thus for such $i_2 \in K_t^1$, \eqref{crossTermOutSupp} is satisfied for $N$ sufficiently large (uniformly in $u$, $v$, $x_1$, and $x_2$; also the choice of how large $N$ must be only depends on $\xi$ and $\mu_V$). The same statement holds for $K_t^2$.

\noindent We now proceed to bound \eqref{crosscont} for $j_1 = 1$ and $j_2 = 2$ by splitting $J_2$ into the regions $K_t^1$, $K_t^2$ and $J_2 \backslash K_t$. We start with $K_t^1$ (the argument for $K_t^2$ is identical). Our observations from the previous paragraph along with \eqref{simplerexsup} gives:
\begin{multline*}
\int_T\,du \,dv\,dt\bigg|N^2 
\sum_{\substack{ i_1 \in J_1 \\  i_2 \in K_t^1}} 
\int_{\gamma_{i_1}}^{\gamma_{i_1+1}} \,d\mu_V(x_1) 
\int_{\gamma_{i_2}}^{\gamma_{i_2+1}} \,d\mu_V(x_2) 
\Big\{(\lambda_{i_1} -x_1) (\lambda_{i_2} -x_2) t(1-t) \\
\times \Xi^{-1}(\tilde{f})^{(3)} \big( tv(\lambda_{i_1} - x_1) + ut(x_2 - \lambda_{i_2}) + u(\lambda_{i_2} - x_2) + t(x_1 - x_2)+ x_2\big) \Big\}\bigg| \\
  \leq \int_0^1 \sum_{\substack{i_1 \in J_1 \\ i_2 \in K^1_t}} \frac{C N^{2\xi - 2 -\alpha} t(1-t)}{(t(\gamma_{i_1} - \gamma_{i_2}) + (\gamma_{i_2} - E))^4}\,dt \leq \int_0^1 \sum_{i_2 \in K^1_t} \frac{CN^{2\xi - 1 - 2\alpha} t(1-t)}{\left((1-t)(\gamma_{i_2} - E)- \frac{t2M}{N^\alpha} \right)^4}\,dt
\end{multline*}
where in the final line we used $|J_1| \leq CN^{1-\alpha}$ from \Cref{regions} \ref{regionsCard}. Next, note that
\begin{multline*}
\frac{1}{N} \sum_{i_2 \in K^1_t}\frac{1}{((1-t)(\gamma_{i_2} - E)- \frac{t2M}{N^\alpha} )^4} \\
\leq C\int_{E + \frac{2M}{N^\alpha}\left(\frac{1 + t}{1-t}\right)}^{E + \frac{1}{2}(E-a)\wedge (b-E)} \frac{dx}{((1-t)(x - E) - \frac{t 2M}{N^\alpha})^4} \leq \frac{CN^{3\alpha}}{1-t},
\end{multline*}
since, by definition of $K_t^1$, $\gamma_{i_2} \geq E + \frac{2M}{N^\alpha}\left(\frac{1+t}{1-t}\right)$. We conclude,
\begin{equation*}
\int_0^1 \sum_{i_2 \in K^1_t} \frac{CN^{2\xi - 1 - 2\alpha} t(1-t)}{\left((1-t)(\gamma_{i_2} - E)- \frac{t2M}{N^\alpha} \right)^4}\,dt\leq C N^{2\xi + \alpha}.
\end{equation*}

\noindent We continue with $J_2 \backslash K_t$. By the same argument as in \Cref{regions} \ref{regionsCard} $|J_2 \backslash K_t | \leq \frac{C N^{1-\alpha}}{1-t}$ where the constant $C$ does not depend on $t$, we use this in addition with \Cref{bornebounds} \cref{borneexsup}, $|J_1| \leq C N^{1-\alpha}$, and $|\lambda_{i_j} - x_j| \leq CN^{\xi - 1}$ to obtain the bound
\begin{multline*}
\int_T\,du \,dv\,dt\bigg|N^2 
\sum_{\substack{ i_1 \in J_1 \\  i_2 \in J_2\backslash K_t}} 
\int_{\gamma_{i_1}}^{\gamma_{i_1+1}} \,d\mu_V(x_1) 
\int_{\gamma_{i_2}}^{\gamma_{i_2+1}} \,d\mu_V(x_2) 
\Big\{(\lambda_{i_1} -x_1) (\lambda_{i_2} -x_2) t(1-t) \\
\times \Xi^{-1}(\tilde{f})^{(3)} \big( tv(\lambda_{i_1} - x_1) + ut(x_2 - \lambda_{i_2}) + u(\lambda_{i_2} - x_2) + t(x_1 - x_2)+ x_2\big) \Big\}\bigg| \\
  \leq C\int_0^1 N^{3\alpha} \log N \times N^{2\xi - 2} \times N^{2 - 2 \alpha}  t \,dt \leq C N^{\alpha + 2 \xi} \log N.
\end{multline*}

\noindent Combining the bounds we have obtained gives
\begin{equation*}
\left|\iint \frac{\Xi^{-1}\tilde{f}(x) - \Xi^{-1}\tilde{f}(y)}{x-y} dM_N^{(1)}(x) dM_N^{(2)}(y)\right| \leq  C N^{\alpha + 2\xi} \log N,
\end{equation*}
which is $\omega(1)$ when divided by $N$  for $\xi$ small enough.

\noindent \textbf{For $j_1 = 1$ or $2$ and $j_2 = 3$.} the proof  is similar and we omit the details.
\end{proof}

\subsection{Proof of \Cref{theoreme}}
\noindent We proceed with the proof of \Cref{theoreme}. Applying the loop equation \eqref{loop1} to $h= \Xi^{-1}(\tilde{f})$  yields
\begin{multline*}
F_1^N\left(\Xi^{-1}(\tilde{f})\right) = \\
M_N (\tilde{f})  + \left(1 -\frac{\beta}{2}\right) L_N\left( (\Xi^{-1} \tilde{f})'\right) +  \frac{1}{N} \left[\frac{\beta}{2}\iint \frac{\Xi^{-1}\tilde{f}(x) - \Xi^{-1}\tilde{f}(y)}{x-y} dM_N(x) dM_N(y)\right].
\end{multline*}
Combining   \Cref{controle2} \cref{mean} and \cref{cross}  we get 
\begin{equation}\label{rang1}
F_1^N(\Xi^{-1}(\tilde{f}))= M_N(\tilde{f}) + \omega(1) .
\end{equation}
\noindent Using the first loop equation from \Cref{loopequations}, and the fact that $M_N(\tilde{f})$ is bounded by $2 N \|f\|_{C(\mathbb{R})}$ gives
\begin{equation}\label{esp1}
\mathbb{E}_V^N \big(M_N(\tilde{f}) \big) = o(1) .
\end{equation}

\noindent We now show recursively that 
\begin{equation}
F_k^N(\Xi^{-1}(\tilde{f}), \tilde{f}, \cdots , \tilde{f} )= \tilde{M}_N(\tilde{f})^k - (k-1) \sigma^2_f \  \tilde{M}_N(\tilde{f})^{k-2}+ \omega(1) .
\end{equation}
Here, the set on which the bound holds might vary from one $k$ to another but each bound has probability greater than $1 - e^{-N^{c_k}}$ for each fixed $k$. 

\noindent The bound holds for $k=1$, by \eqref{rang1}. Now, assume this holds for $k \geq 1$. Then by Proposition \ref{loopequations} we have
\begin{equation}
 F^N_{k+1}(\Xi^{-1}(\tilde{f}), \tilde{f}, \cdots , \tilde{f})  = F^N_k(\Xi^{-1}(\tilde{f}), \tilde{f}, \cdots , \tilde{f}) \tilde{M}_N(\tilde{f}) +  \tilde{M}_N(\tilde{f})^{k-1} \  L_N(\Xi^{-1}(\tilde{f}) \tilde{f}')
\end{equation}
On a set of probability greater than  $1 - e^{-N^{c_{k+1}}}$ we have by the induction hypothesis,  \Cref{controle} \cref{controle1}, and  \Cref{controle2} \cref{variance}, for some $\delta > 0$ and a constant $C$
\begin{equation*}
\big| F_k^N(\Xi^{-1}(\tilde{f}), \tilde{f}, \cdots , \tilde{f} ) - \tilde{M}_N(\tilde{f})^k - (k-1) \sigma^2_f \  \tilde{M}_N(\tilde{f})^{k-2} \big| \leq \frac{C}{N^\delta},
\end{equation*}
 \begin{equation*}
\big| L_N \Big(\Xi^{-1}(\tilde{f}) \tilde{f}'  \Big) + \sigma_f^2  \big| \leq \frac{C}{N^\delta} \ , 
\end{equation*}
\begin{equation*}
\big|M_N(\tilde{f})\big| \leq N^{\delta/2k}  \ .   
\end{equation*}
\noindent And this proves the induction. Using the fact that $F_k$ is bounded polynomially and deterministically, the computation of the moments is then straightforward and this concludes the proof of Theorem \ref{theoreme}.

\begin{rem}
The same proof would  also  show the macroscopic central limit Theorem already shown in \cite{BGII,ShI,J} but with  less restrictive condition $V \in C^6(\mathbb{R})$ and $f \in C^5(\mathbb{R})$ with appropriate decay conditions. 
\end{rem}

\bibliographystyle{abbrv}
\bibliography{biblio}

\end{document}